\documentclass{amsart}        
 \newtheorem{thm}{Theorem}[section]

 \theoremstyle{definition}
 \newtheorem{defn}[thm]{Definition}
 \theoremstyle{remark}
 \newtheorem{rem}[thm]{Remark}
 
 \numberwithin{equation}{section}

\newcommand\refthm[1]{\ref{thm:#1}}

\newcommand\eqn[1]{(\ref{eq:#1})}
\newcommand\sect[1]{\ref{sec:#1}}

\newcommand{\sgn}{\textnormal{sgn}}
\newcommand{\beqs}{\begin{equation*}}
\newcommand{\eeqs}{\end{equation*}}
\newcommand{\beq}{\begin{equation}}
\newcommand{\eeq}{\end{equation}}
\newcommand{\qbas}[5]{{}_2\phi_1\left(\begin{matrix} 
    #1, &#2; &#3, &#4 \\ 
        &#5  &    & 
\end{matrix}\right)}
\newcommand{\Parans}[1]{\left(#1\right)}
\newcommand{\aqprod}[3]{\Parans{#1;#2}_{#3}}
\newcommand{\abs}[1]{\lvert#1\rvert}
\newcommand\leg[2]{\genfrac{(}{)}{}{}{#1}{#2}} 
\allowdisplaybreaks

\begin{document}

%
%
%
%
%
%
%
%
%

\title[Mock theta function identities]
{New fifth and seventh order mock theta \\
function identities}

\author{F. G. Garvan}
\address{Department of Mathematics, University of Florida, Gainesville,
FL 32611-8105}
\email{fgarvan@ufl.edu}
\thanks{The author was supported in part by a grant from 
the Simon's Foundation (\#318714).}
\subjclass[2010]{Primary 33D15; Secondary 11B65, 11F27}             

\keywords{Mock theta functions, Hecke-Rogers double sums, Bailey pairs, conjugate
Bailey pairs}

\date{October 9, 2018}
\dedicatory{Dedicated to George Andrews on the occasion of his eightieth birthday}

\begin{abstract}
We give simple proofs of Hecke-Rogers indefinite binary theta series identities 
for the two Ramanujan
fifth order mock theta functions $\chi_0(q)$ and $\chi_1(q)$ and
all three of Ramanujan's seventh order mock theta functions.
We find that the coefficients of the three mock theta functions of order
$7$ are surprisingly related.
\end{abstract}

\maketitle
\section{Introduction}
\label{sec:intro}

In his last letter to G.~H.~Hardy, Ramanujan described new functions that
he called mock theta functions and listed mock theta functions of order
$3$, $5$ and $7$. Watson studied the behaviour of the third order
functions under the modular group, but was unable to find similar transformation
properties for the fifth and seventh order functions. The first substantial
progress towards finding such transformation properties was made by       
Andrews \cite{An86},
who found double sum representations for the fifth and seventh order
functions. These double sum representations were reminiscent of certain
identities for modular forms found by Hecke and Rogers. Andrews results
for the fifth and seventh order mock theta functions were crucial to 
Zwegers \cite{Zw-thesis},
who later showed how to complete these functions to harmonic Maass forms. 
For more details on this aspect see Zagier's survey \cite{Zag06}.

Throughout this paper we use following standard notation:
$$
(a;q)_\infty = (a)_\infty = \prod_{n=1}^\infty (1-aq^n),
$$
\begin{align*}
(a;q)_n &= (a)_n = (a;q)_\infty /  (a q^n;q)_\infty\\
&\left(=(1-a)(1 -aq) \cdots (1- a q^{n-1})\quad
\mbox{for $n$ a nonnegative integer}\right).
\end{align*}

Andrews \cite{An86} 
found Hecke-Rogers indefinite binary theta series  identities for all the fifth order
mock theta functions except for the following two: 
\begin{align*}
\chi_0(q) &= \sum_{n=0}^\infty \frac{q^n}{(q^{n+1};q)_n}
           = \sum_{n=0}^\infty \frac{q^n \,(q)_n}{(q)_{2n}}\\
&=1+q+{q}^{2}+2\,{q}^{3}+{q}^{4}+3\,{q}^{5}+2\,{q}^{6}+3\,{q}^{7}+\cdots,
\end{align*}
and
\begin{align*}
\chi_1(q) &= \sum_{n=0}^\infty \frac{q^n}{(q^{n+1};q)_{n+1}}
           = \sum_{n=0}^\infty \frac{q^n \,(q)_n}{(q)_{2n+1}}\\
&=1+2\,q+2\,{q}^{2}+3\,{q}^{3}+3\,{q}^{4}+4\,{q}^{5}+4\,{q}^{6}+
6\,{q}^{7}+\cdots.
\end{align*}
Zwegers \cite{Zw09} found triple
sum identities for $\chi_0(q)$ and $\chi_1(q)$.  Zagier \cite{Zag06}
stated indefinite binary theta series identities for these two functions but gave 
few details.  We find new
Hecke-Rogers indefinite binary theta series identities for these two functions.
In Section \sect{zagier} we compare our results with Zagier's.

\begin{thm}
\label{thm:chi01}
\begin{align}
&(q)_\infty (\chi_0(q) - 2) \nonumber\\ 
&=\sum_{j=0}^\infty 
\sum_{\substack{-j \le 3m \le j}}
\sgn(m) (-1)^{m+j+1} q^{ j(3j+1)/2 - m(15m+1)/2 }(1-q^{2j+1})
\label{eq:chi01a}\\
&+
\sum_{j=1}^\infty 
\sum_{\substack{-j-1 \le 3m \le j-1}}
\sgn(m) (-1)^{m+j+1} q^{ j(3j+1)/2 - m(15m+11)/2 -1}(1-q^{2j+1}),
\nonumber
\end{align}
and
\begin{align}
&(q)_\infty \chi_1(q) \nonumber\\
&=\sum_{j=1}^\infty 
\sum_{\substack{-j \le 3m \le j-1}}
\sgn(m) (-1)^{m+j+1} q^{ j(3j-1)/2 - m(15m+7)/2 -1}(1+q^{j})
\label{eq:chi01b}\\
&+
\sum_{j=1}^\infty 
\sum_{\substack{-j-1 \le 3m \le j-2}}
\sgn(m) (-1)^{m+j+1} q^{ j(3j-1)/2 - m(15m+13)/2 -2}(1+q^{j}),
\nonumber
\end{align}
where
$$
\sgn(m) =\begin{cases} 1 & \mbox{if $m\ge0$},\\
                       -1 & \mbox{if $m<0$}. \end{cases}
$$
\end{thm}

\textit{Idea of Proof}. 
We need the following conjugate Bailey pair (with $a=q$):
\begin{align*}
\delta_n &= \frac{q^n (q)_n (q)_\infty}{(1-q)},\\
\gamma_n &= \sum_{j=n+1}^\infty (-1)^{j+n+1} q^{j(3j-1)/2 - 3n(n+1)/2 -1}(1+q^j).
\end{align*}
The proof of this only uses Heine's transformation \cite[Eq.(III.I)]{Ga-Ra-book}
and an exercise from 
Andrews's book \cite[Ex.10,p.29]{An-book}.
The rest of the proof of Theorem \refthm{chi01} uses this conjugate Bailey pair, 
the Bailey transform
and Slater's Bailey pairs A(4) and A(2) (with $a=q$) \cite[p.463]{Sl51}.
The necessary background on conjugate Bailey pairs, Bailey pairs and
the Bailey transform is given in Section \sect{bailey}.
In Section \sect{proofthm1} the proof of Theorem \refthm{chi01} is completed.

Using the same conjugate Bailey pair and Slater's A(7*), A(8) and A(6) (with
$a=q$) lead to new 
Hecke-Rogers indefinite binary theta series identities 
for Ramanujan's three seventh order mock theta functions. 
A(7*) is actually a variant
of A(7) adjusted to work with $a=q$ instead of $a=1$. The three identities
given below in Theorem \refthm{mock7} appear to be new.                      
The following are Ramanujan's three seventh order mock theta functions:
\begin{align*}
\mathcal{F}_0(q) &= \sum_{n=0}^\infty \frac{q^{n^2}}{(q^{n+1};q)_n}\\
&=1+q+{q}^{3}+{q}^{4}+{q}^{5}+2\,{q}^{7}+{q}^{8}+2\,{q}^{9}+\cdots,\\
\mathcal{F}_1(q) &= \sum_{n=1}^\infty \frac{q^{n^2}}{(q^{n};q)_n}\\
&=q+{q}^{2}+{q}^{3}+2\,{q}^{4}+{q}^{5}+2\,{q}^{6}+2\,{q}^{7}+2\,{q}^{8}+\cdots,\\
\mathcal{F}_2(q) &= \sum_{n=0}^\infty \frac{q^{n^2+n}}{(q^{n+1};q)_{n+1}}\\ 
&=1+q+2\,{q}^{2}+{q}^{3}+2\,{q}^{4}+2\,{q}^{5}+3\,{q}^{6}+2\,{q}^{7}+\cdots. 
\end{align*}

We have the following theorem
\begin{thm}
\label{thm:mock7}
\begin{align}
&(q)_\infty \mathcal{F}_0(q) \nonumber\\
&=\sum_{j=1}^\infty 
\sum_{\substack{-j \le 3m \le j-1}}
\sgn(m) (-1)^{m+j+1} q^{ j(3j-1)/2 - m(21m+13)/2 -1}\nonumber\\
&{\hphantom{XXXXXXXXXXXXXXXX}}(1+q^{j})(1-q^{6m+1}),
\label{eq:F0id}\\
&(q)_\infty \mathcal{F}_1(q) \nonumber\\
&=\sum_{j=1}^\infty 
\sum_{\substack{-j \le 3m \le j-1}}
\sgn(m) (-1)^{m+j+1} q^{ j(3j-1)/2 - m(21m+5)/2}(1+q^{j})
\label{eq:F1id}\\
& + 
\sum_{j=2}^\infty 
\sum_{\substack{-j-1 \le 3m \le j-2}}
\sgn(m) (-1)^{m+j+1} q^{ j(3j-1)/2 - m(21m+19)/2 -2}(1+q^{j})\nonumber\\
&(q)_\infty \mathcal{F}_2(q) \nonumber\\
&=\sum_{j=1}^\infty 
\sum_{\substack{-j \le 3m \le j-1}}
\sgn(m) (-1)^{m+j+1} q^{ j(3j-1)/2 - m(21m+11)/2 -1}(1+q^{j})
\label{eq:F2id}\\
& + 
\sum_{j=2}^\infty 
\sum_{\substack{-j-1 \le 3m \le j-2}}
\sgn(m) (-1)^{m+j+1} q^{ j(3j-1)/2 - m(21m+17)/2 -2}(1+q^{j}).
\nonumber
\end{align}
\end{thm}




We prove this theorem in Section \sect{proofthm2}.
In his last letter to Hardy, all that Ramanujan said about the seventh order 
functions was that there were not related to each other.
Surprisingly we show that the coefficients of the three seventh order
functions are indeed related, although this is probably not the kind of
relationship that Ramanujan had in mind. For example we find for $n\ge0$ that
\begin{align}
f_0(25 n + 8) &= f_2(n),
\label{eq:f025}\\
f_1(25 n + 1) &= f_0(n),
\label{eq:f125}\\
f_2(25 n - 3) &= -f_1(n),
\label{eq:f225}
\end{align}
where 
we define $f_j(n)$ by
$$                    
\sum_{n=0}^\infty f_j(n) q^n = (q)_\infty \,\mathcal{F}_j(q),
$$
for $j=0$, $1$, $2$. This and more general results including analogous
results for the fifth order functions are proved in Section \sect{zagier}.

\section{The Bailey Transform and Conjugate Bailey Pairs}
\label{sec:bailey}

\begin{thm}[The Bailey Transform]
\label{thm:baileytrans}
Subject to suitable convergence conditions, if
\beq
\beta_n = \sum_{r=0}^n \alpha_r u_{n-r} v_{n+r},\quad\mbox{and}
\qquad
\gamma_n = \sum_{r=n}^\infty \delta_r u_{r-n} v_{r+n},
\label{eq:bgdef}
\eeq
then
\beq
\sum_{n=0}^\infty \alpha_n \gamma_n = \sum_{n=0}^\infty \beta_n \delta_n.
\label{eq:bsum}
\eeq
\end{thm}

When applying his transform, Bailey \cite{Ba48} 
chose $u_n = 1/(q)_n$ and $v_n=1/(aq;q)_n$. This motivates the following
definitions:
\begin{defn}
A pair of sequences $(\alpha_n,\beta_n)$ is a \textbf{Bailey pair}
relative to $(a,q)$ if
\beq
\beta_n = \sum_{r=0}^n \frac{\alpha_r}{(q)_{n-r} (aq)_{n+r}},
\eeq
for $n\ge0$.
\end{defn}
\begin{defn}
A pair of sequences $(\gamma_n,\delta_n)$ is a \textbf{conjugate Bailey pair}
relative to $(a,q)$ if
\beq
\gamma_n = \sum_{r=n}^\infty \frac{\delta_r}{(q)_{r-n} (aq)_{r+n}},
\eeq
for $n\ge0$.
\end{defn}

The basic idea is to find a suitable conjugate Bailey pair and apply
the Bailey Transform using known Bailey pairs.

\begin{thm}
\label{thm:conjbpair}
The sequences
\begin{align}
\delta_n &= \frac{q^n (q)_n (q)_\infty}{(1-q)},
\label{eq:dn}\\
\gamma_n &= \sum_{j=n+1}^\infty (-1)^{j+n+1} q^{j(3j-1)/2 - 3n(n+1)/2 -1}(1+q^j),
\label{eq:gn}
\end{align}
form a conjugate Bailey pair relative to $(q,q)$; i.e.\ $a=q$.
\end{thm}
\begin{rem}
We note that this result can be deduced from a special case of a result of Lovejoy
\cite[Thm1.1(4),p.53]{Lo12}.
We give a simple proof that uses only Heine's transformation and
a combinatorial result of Andrews \cite[Ex.10,p.29]{An-book}.
\end{rem}
\begin{proof}
We let 
\beqs
\delta_n = (q)_n (q)_\infty \frac{q^n}{(1-q)},
\eeqs
and
\beqs
\gamma_n = \sum_{r=n}^\infty \frac{\delta_r}{(q)_{r-n} (q^2;q)_{r+n}}
= (q)_\infty \sum_{r=n}^\infty \frac{(q)_r q^r}{(q)_{r-n} (q;q)_{r+n+1}}.
\eeqs
We must show that $\gamma_n$ is given by \eqn{gn}.
\begin{align*}
&\sum_{r=n}^\infty \frac{(q)_r q^r}{(q)_{r-n} (q;q)_{r+n+1}} 
=\sum_{r=0}^\infty \frac{(q)_{r+n} q^{r+n}}{(q)_{r} (q)_{r+2n+1}}\\
&= q^n\frac{(q)_n}{(q)_{2n+1}}
\sum_{r=0}^\infty \frac{(q^{n+1};q)_{r} q^{r}}{(q)_{r} (q^{2n+2};q)_{r}}
= q^n\frac{(q)_n}{(q)_{2n+1}}
\qbas{0}{q^{n+1}}{q}{q}{q^{2n+2}}\\
&= q^n\frac{(q)_n}{(q)_{2n+1}} \frac{(q^{n+1};q)_\infty}{(q^{2n+2};q)_\infty (q)_\infty} \, \qbas{q^{n+1}}{q}{q}{q^{n+1}}{0}\\
&=q^n \frac{1}{(q)_\infty} \sum_{j=0}^\infty (q^{n+1};q)_j q^{(n+1)j},
\end{align*}
by Heine's transformation \cite[Eq.(III.I)]{Ga-Ra-book}, so that
\beq
\gamma_n = q^n \sum_{j=0}^\infty (q^{n+1};q)_j q^{(n+1)j}.
\label{eq:gnb}
\eeq
From Andrews \cite[Ex.10,p.29]{An-book}
we have
\beq
\sum_{j=0}^\infty (xq)_j x^{j+1} q^{j+1}
=\sum_{m=1}^\infty (-1)^{m-1} q^{m(3m-1)/2} x^{3m-2}(1 + x q^m).
\label{eq:andex10}
\eeq
Using \eqn{gnb} and \eqn{andex10} with $x=q^n$ we have
\begin{align*}
\gamma_n &= q^n \sum_{j=0}^\infty (q^{n+1};q)_j q^{(n+1)j} \\
&=\sum_{m=1}^\infty (-1)^{m-1} q^{m(3m-1)/2 + n(3m-2)-1}(1+q^{m+n})\\
&=\sum_{m=n+1}^\infty (-1)^{m+n+1} q^{m(3m-1)/2 - 3n(n+1)/2 -1}(1+q^{m}),
\end{align*}
as required.
We note that Subbarao \cite{Su71} gave
a combinatorial proof of \eqn{andex10} by using a variant of Franklin's 
involution \cite[pp.10--11]{An-book}.
\end{proof}

\section{Proof of Theorem \refthm{chi01}}
\label{sec:proofthm1}

To prove Theorem \refthm{chi01} we will apply the Bailey Transform,
with $u_n= 1/(q)_n$, $v_n=1/(q^2;q)_n$, using the conjugate Bailey pair
in Theorem \refthm{conjbpair}, and Slater's Bailey pairs $A(4)$ and $A(2)$.
By \cite[p.463]{Sl51}, the following gives Slater's $A(4)$ Bailey pair
relative to $(q,q)$:
\beq   
\beta_n = \frac{q^n}{\aqprod{q^2}{q}{2n}}
        ,\qquad
        \alpha_n =
        \left\{\begin{array}{ll}
                 q^{6m^2-4m}& \mbox{ if } n = 3m-1,
                \\
                q^{6m^2+ 4m}& \mbox{ if } n = 3m,
                \\
                -q^{6m^2 + 8m + 2}-q^{6m^2+4m}& \mbox{ if } n = 3m+1.
        \end{array}\right.
\label{eq:A4}
\eeq
By \cite[Eq.(A${}_0$),p.278]{Wa36b} we have
\begin{align*}
\chi_0(q) &= \sum_{n=0}^\infty \frac{q^n}{(q^{n+1};q)_n}
          =1 +  \sum_{n=0}^\infty \frac{q^{2n+1}}{\aqprod{q^{n+1}}{q}{n+1}}\\
&=1 + q \sum_{n=0}^\infty
      \frac{q^n}{\aqprod{q^2}{q}{2n}}\cdot \frac{q^n(q)_n}{(1-q)} 
=1 + \frac{q}{(q)_\infty} \sum_{n=0}^\infty \beta_n \delta_n,
\end{align*}
where $\delta_n$ is given in \eqn{dn}.
Thus by the Bailey Transform and \eqn{A4} we have
\begin{align}
&(q)_\infty \Parans{\chi_0(q) - 1} 
= q\sum_{n=0}^\infty \beta_n \delta_n
= q\sum_{n=0}^\infty \alpha_n \gamma_n
\label{eq:chi0id1}\\
&=\sum_{m=1}^\infty \sum_{j=3m}^\infty (-1)^{m+j} q^{j(3j-1)/2 - m(15m-1)/2}
 (1+q^j) \nonumber\\
& \qquad
+\sum_{m=0}^\infty \sum_{j=3m+1}^\infty (-1)^{m+j+1} q^{j(3j-1)/2 - m(15m+1)/2}
 (1+q^j) \nonumber\\
& \qquad
+\sum_{m=0}^\infty \sum_{j=3m+2}^\infty (-1)^{m+j+1} 
      \left\{ q^{j(3j-1)/2 - m(15m+11)/2-1} \right.\nonumber\\
&\left.\hphantom{XXXXXXXXXXXXXXX} + q^{j(3j-1)/2 - m(15m+19)/2-3}\right\} (1+q^j)\nonumber\\
&=
\sum_{j=1}^\infty 
\sum_{\substack{-j \le 3m \le j-1}}
\sgn(m) (-1)^{m+j+1} q^{ j(3j-1)/2 - m(15m+1)/2 }(1+q^{j})\nonumber\\
& \qquad +
\sum_{j=2}^\infty 
\sum_{\substack{-j-1 \le 3m \le j-2}}
\sgn(m) (-1)^{m+j+1} q^{ j(3j-1)/2 - m(15m+11)/2 -1}(1+q^{j}),
\nonumber
\end{align}
by noting that 
$$
-(-m-1)(15(-m-1)+19)/2-3 = -m(15m+11)/2-1.
$$
Now from Euler's Pentagonal Number Theorem \cite[p.11]{An-book} we have
\beq
(q)_\infty = \sum_{n=-\infty}^\infty (-1)^n q^{n(3n-1)/2}
=\sum_{m=-\infty}^\infty q^{6m^2+m} - \sum_{m=-\infty}^\infty q^{6m^2+5m + 1}.
\label{eq:EPT}
\eeq
By \eqn{chi0id1} and \eqn{EPT} we have
\begin{align*}
&(q)_\infty \Parans{\chi_0(q) - 2} =
(q)_\infty \Parans{\chi_0(q) - 1} - (q)_\infty\\
&=\sum_{j=1}^\infty 
\sum_{\substack{-j \le 3m \le j-1}}
\sgn(m) (-1)^{m+j+1} q^{ j(3j-1)/2 - m(15m+1)/2 }(1+q^{j})\\
& \qquad +
\sum_{j=1}^\infty 
\sum_{\substack{-j-1 \le 3m \le j-2}}
\sgn(m) (-1)^{m+j+1} q^{ j(3j-1)/2 - m(15m+11)/2 -1}(1+q^{j}),\\
&\qquad -\sum_{m=-\infty}^\infty q^{6m^2+m} + \sum_{m=-\infty}^\infty q^{6m^2+5m + 1}\\
&=\sum_{j=1}^\infty 
\sum_{\substack{-j+1 \le 3m \le j-1}}
\sgn(m) (-1)^{m+j+1} q^{ j(3j-1)/2 - m(15m+1)/2 }\\
& \qquad +
 \sum_{j=0}^\infty 
\sum_{\substack{-j \le 3m \le j}}
\sgn(m) (-1)^{m+j+1} q^{ j(3j+1)/2 - m(15m+1)/2 }\\
& \qquad+
\sum_{j=2}^\infty 
\sum_{\substack{-j \le 3m \le j-2}}
\sgn(m) (-1)^{m+j+1} q^{ j(3j-1)/2 - m(15m+11)/2 -1} \\
& \qquad +
\sum_{j=1}^\infty 
\sum_{\substack{-j-1 \le 3m \le j-1}}
\sgn(m) (-1)^{m+j+1} q^{ j(3j+1)/2 - m(15m+11)/2 -1}.
\end{align*}
On the right side of the last equation above replace $j$ by $j+1$ in 
the first and third double sums to obtain
\begin{align*}
&(q)_\infty (\chi_0(q) - 2) \\ 
&=\sum_{j=0}^\infty 
\sum_{\substack{-j \le 3m \le j}}
\sgn(m) (-1)^{m+j+1} q^{ j(3j+1)/2 - m(15m+1)/2 }(1-q^{2j+1})
\\
&+
\sum_{j=1}^\infty 
\sum_{\substack{-j-1 \le 3m \le j-1}}
\sgn(m) (-1)^{m+j+1} q^{ j(3j+1)/2 - m(15m+11)/2 -1}(1-q^{2j+1}),
\end{align*}
which is \eqn{chi01a}.

To prove \eqn{chi01b} we need Slater's \cite[p.463]{Sl51} A(2) Bailey pair relative to $(q,q)$:
\beq   
\beta_n = \frac{1}{\aqprod{q^2}{q}{2n}}
        ,\qquad
        \alpha_n =
        \left\{\begin{array}{ll}
                 q^{6m^2-m}& \mbox{ if } n = 3m-1,
                \\
                q^{6m^2+ m}& \mbox{ if } n = 3m,
                \\
                -q^{6m^2 + 5m + 1}-q^{6m^2+7m+2}& \mbox{ if } n = 3m+1.
        \end{array}\right.
\label{eq:A2}
\eeq
We have
\begin{align*}
\chi_1(q) &= \sum_{n=0}^\infty \frac{q^n}{(q^{n+1};q)_{n+1}}
          =\sum_{n=0}^\infty \frac{q^{n}(q)_n}{(q)_{2n+1}}                \\
&=\sum_{n=0}^\infty
      \frac{1}{\aqprod{q^2}{q}{2n}}\cdot \frac{q^n(q)_n}{(1-q)} 
=\frac{1}{(q)_\infty} \sum_{n=0}^\infty \beta_n \delta_n.
\end{align*}
By the Bailey Transform and \eqn{A2} we have
\begin{align}
&(q)_\infty \chi_1(q) 
=\sum_{n=0}^\infty \beta_n \delta_n
= \sum_{n=0}^\infty \alpha_n \gamma_n
\label{eq:chi1id1}\\
&=\sum_{m=1}^\infty \sum_{j=3m}^\infty (-1)^{m+j} q^{j(3j-1)/2 - m(15m-7)/2-1}
 (1+q^j) \nonumber\\
& \qquad
+\sum_{m=0}^\infty \sum_{j=3m+1}^\infty (-1)^{m+j+1} q^{j(3j-1)/2 - m(15m+7)/2-1}
 (1+q^j) \nonumber\\
& \qquad
+\sum_{m=0}^\infty \sum_{j=3m+2}^\infty (-1)^{m+j+1} 
      \left\{ q^{j(3j-1)/2 - m(15m+17)/2-3}\right.\nonumber\\
&\hphantom{XXXXXXXXXXXX}\left.+ q^{j(3j-1)/2 - m(15m+13)/2-2}\right\} (1+q^j)\nonumber\\
&=
\sum_{j=1}^\infty 
\sum_{\substack{-j \le 3m \le j-1}}
\sgn(m) (-1)^{m+j+1} q^{ j(3j-1)/2 - m(15m+7)/2 -1}(1+q^{j})\nonumber\\
& \qquad +
\sum_{j=1}^\infty 
\sum_{\substack{-j-1 \le 3m \le j-2}}
\sgn(m) (-1)^{m+j+1} q^{ j(3j-1)/2 - m(15m+13)/2 -2}(1+q^{j}),
\nonumber
\end{align}
by noting that
$$
-(-m-1)(15(-m-1)+17)/2-3 = -m(15m+13)/2-2.
$$
This completes the proof of Theorem \refthm{chi01}.

\section{Proof of Theorem \refthm{mock7}}
\label{sec:proofthm2}
To prove Theorem \refthm{mock7} we proceed as in Section \sect{proofthm1}.
This time we need Slater's Bailey pairs $A(6)$ and $A(8)$, and a variant
of her Bailey pair $A(7)$.

From \cite[Eq.(3.4),p.464]{Sl51} we have
\begin{align*}
&\frac{q^{n^2-n}}
       {(q)_{2n}} 
= \sum_{r=-[(n+1)/3]}^{[n/3]} 
\frac{(1-q^{6r+1})q^{3r^2-2r}}
     {(q)_{n+3r+1} (q)_{n-3r}}\\
&=
 \sum_{r=0}^{[n/3]} \frac{(1-q^{6r+1})q^{3r^2-2r}}
                                {(q)_{n-3r} (q)_{n+3r+1}}
+
 \sum_{r=1}^{[(n+1)/3]} 
\frac{(1-q^{-6r+1})q^{3r^2+2r}}
     {(q)_{n+3r} (q)_{n+1-3r}}
\\
&=
 \sum_{r=1}^{[(n+1)/3]} \frac{q^{3r^2+2r}-q^{3r^2-4r+1}}
                                {(q)_{n-(3r-1)} (q)_{n+(3r-1)+1}}
+
 \sum_{r=0}^{[n/3]} \frac{q^{3r^2-2r}-q^{3r^2+4r+1}}
                                {(q)_{n-3r} (q)_{n+3r+1}},
\end{align*}
so that
\begin{align*}
&\frac{(1-q)q^{n^2-n}}
       {(q)_{2n}} \\
&=
 \sum_{r=1}^{[(n+1)/3]} \frac{q^{3r^2+2r}-q^{3r^2-4r+1}}
                                {(q)_{n-(3r-1)} \aqprod{q^2}{q}{n+(3r-1)}}
+
 \sum_{r=0}^{[n/3]} \frac{q^{3r^2-2r}-q^{3r^2+4r+1}}
                                {(q)_{n-3r}\aqprod{q^2}{q}{n+3r}}.
\end{align*}
This implies the following Bailey pair relative to $(q,q)$:
\beq   
\beta_n = \frac{(1-q)q^{n^2-n}}{(q)_{2n}}
        ,\qquad
        \alpha_n =
        \left\{\begin{array}{ll}
                 q^{3m^2+2m}-q^{3m^2-4m+1}& \mbox{ if } n = 3m-1,
                \\
                q^{3m^2-2m}-q^{3m^2+4m+1}& \mbox{ if } n = 3m,
                \\
                0 & \mbox{ if } n = 3m+1.
        \end{array}\right.
\label{eq:A7star}
\eeq
We note that this Bailey pair was found by Warnaar \cite[p.375]{Wa03}
using a different method.
We have
\begin{align*}
\mathcal{F}_0(q) &= \sum_{n=0}^\infty \frac{q^{n^2}}{\aqprod{q^{n+1}}{q}{n}}\\
&=\sum_{n=0}^\infty
      \frac{(1-q)q^{n^2-n}}{(q)_{2n}}\cdot \frac{q^n(q)_n}{(1-q)} 
=\frac{1}{(q)_\infty} \sum_{n=0}^\infty \beta_n \delta_n.
\end{align*}
Thus by the Bailey Transform and \eqn{A7star} we have
\begin{align}
&(q)_\infty \mathcal{F}_0(q)              
= \sum_{n=0}^\infty \beta_n \delta_n
= \sum_{n=0}^\infty \alpha_n \gamma_n
\label{eq:F0id1}\\
&=\sum_{m=1}^\infty \sum_{j=3m}^\infty (-1)^{m+j} q^{j(3j-1)/2 - m(21m-13)/2-1}
 (1+q^j) \nonumber\\
& \qquad
+\sum_{m=1}^\infty \sum_{j=3m}^\infty (-1)^{m+j+1} q^{j(3j-1)/2 - m(21m-1)/2}
 (1+q^j) \nonumber\\
& \qquad
+\sum_{m=0}^\infty \sum_{j=3m+1}^\infty (-1)^{m+j+1} 
       q^{j(3j-1)/2 - m(21m+13)/2-1} 
(1+q^j)\nonumber\\
& \qquad
+\sum_{m=0}^\infty \sum_{j=3m+1}^\infty (-1)^{m+j} 
       q^{j(3j-1)/2 - m(21m+1)/2} 
(1+q^j)\nonumber\\
&=
\sum_{j=1}^\infty 
\sum_{\substack{-j \le 3m \le j-1}}
\sgn(m) (-1)^{m+j+1} q^{ j(3j-1)/2 - m(21m+13)/2-1 }(1+q^{j})\nonumber\\
& \qquad +
\sum_{j=1}^\infty 
\sum_{\substack{-j \le 3m \le j-1}}
\sgn(m) (-1)^{m+j} q^{ j(3j-1)/2 - m(21m+1)/2}(1+q^{j}),
\nonumber\\
&=
\sum_{j=1}^\infty 
\sum_{\substack{-j \le 3m \le j-1}}
\sgn(m) (-1)^{m+j+1} q^{ j(3j-1)/2 - m(21m+13)/2-1 }\nonumber\\
&\hphantom{XXXXXXXXXXXXXX}(1+q^{j})(1-q^{6m+1}),
\nonumber
\end{align}
which is \eqn{F0id}.

To prove \eqn{F1id} we need Slater's \cite[p.463]{Sl51} A(8) Bailey pair relative to $(q,q)$:
\beq   
\beta_n = \frac{q^{n^2+n}}{\aqprod{q^2}{q}{2n}}
        ,\qquad
        \alpha_n =
        \left\{\begin{array}{ll}
                 q^{3m^2-2m}& \mbox{ if } n = 3m-1,
                \\
                q^{3m^2+ 2m}& \mbox{ if } n = 3m,
                \\
                -q^{3m^2 + 4m + 1}-q^{3m^2+2m}& \mbox{ if } n = 3m+1.
        \end{array}\right.
\label{eq:A8}
\eeq
We have
\begin{align*}
\mathcal{F}_1(q) &= \sum_{n=1}^\infty \frac{q^{n^2}}
                                           {\aqprod{q^n}{q}{n}}
=q\sum_{n=0}^\infty \frac{q^{n^2+2n}}
                         {\aqprod{q^{n+1}}{q}{n+1}}\\
&=q\sum_{n=0}^\infty
      \frac{q^{n^2+n}}
           {\aqprod{q^2}{q}{2n}}\cdot \frac{q^n(q)_n}{(1-q)} 
=\frac{q}{(q)_\infty} \sum_{n=0}^\infty \beta_n \delta_n.
\end{align*}

Thus by the Bailey Transform and \eqn{A8} we have
\begin{align}
&(q)_\infty \mathcal{F}_1(q)              
= \sum_{n=0}^\infty \beta_n \delta_n
= \sum_{n=0}^\infty \alpha_n \gamma_n
\label{eq:F1id1}\\
&=\sum_{m=1}^\infty \sum_{j=3m}^\infty (-1)^{m+j} q^{j(3j-1)/2 - m(21m-5)/2}
 (1+q^j) \nonumber\\
& \qquad
+\sum_{m=0}^\infty \sum_{j=3m+1}^\infty (-1)^{m+j+1} q^{j(3j-1)/2 - m(21m+5)/2}
 (1+q^j) \nonumber\\
& \qquad
+\sum_{m=0}^\infty \sum_{j=3m+2}^\infty (-1)^{m+j+1} \left\{
       q^{j(3j-1)/2 - m(21m+19)/2-2} (1+q^j) \right.\nonumber\\
&\hphantom{XXXXXXXXX}\left.+ q^{j(3j-1)/2 - m(21m+23)/2-3} (1+q^j)\right\}
\nonumber\\
&=
\sum_{j=1}^\infty 
\sum_{\substack{-j \le 3m \le j-1}}
\sgn(m) (-1)^{m+j+1} q^{ j(3j-1)/2 - m(21m+5)/2 }(1+q^{j})\nonumber\\
& \qquad +
\sum_{j=2}^\infty 
\sum_{\substack{-j-1 \le 3m \le j-2}}
\sgn(m) (-1)^{m+j+1} q^{ j(3j-1)/2 - m(21m+19)/2-2}(1+q^{j}),
\nonumber
\end{align}
which is \eqn{F1id}.

To prove \eqn{F2id} we need Slater's \cite[p.463]{Sl51} A(6) Bailey pair relative to $(q,q)$:
\beq   
\beta_n = \frac{q^{n^2}}{\aqprod{q^2}{q}{2n}}
        ,\qquad
        \alpha_n =
        \left\{\begin{array}{ll}
                 q^{3m^2+m}& \mbox{ if } n = 3m-1,
                \\
                q^{3m^2-m}& \mbox{ if } n = 3m,
                \\
                -q^{3m^2 + m}-q^{3m^2+5m + 2}& \mbox{ if } n = 3m+1,
        \end{array}\right.
\label{eq:A6}
\eeq
We have
$$
\mathcal{F}_2(q) = \sum_{n=1}^\infty \frac{q^{n^2+n}}
                                           {\aqprod{q^{n+1}}{q}{n+1}}
=\sum_{n=0}^\infty
      \frac{q^{n^2}}
           {\aqprod{q^2}{q}{2n}}\cdot \frac{q^n(q)_n}{(1-q)} 
=\frac{1}{(q)_\infty} \sum_{n=0}^\infty \beta_n \delta_n.
$$                       

Thus by the Bailey Transform and \eqn{A6} we have
\begin{align}
&(q)_\infty \mathcal{F}_2(q)              
= \sum_{n=0}^\infty \beta_n \delta_n
= \sum_{n=0}^\infty \alpha_n \gamma_n
\label{eq:F2id1}\\
&=\sum_{m=1}^\infty \sum_{j=3m}^\infty (-1)^{m+j} q^{j(3j-1)/2 - m(21m-11)/2-1}
 (1+q^j) \nonumber\\
& \qquad
+\sum_{m=0}^\infty \sum_{j=3m+1}^\infty (-1)^{m+j+1} q^{j(3j-1)/2 - m(21m+11)/2-1}
 (1+q^j) \nonumber\\
& \qquad
+\sum_{m=0}^\infty \sum_{j=3m+2}^\infty (-1)^{m+j+1} \left\{
       q^{j(3j-1)/2 - m(21m+25)/2-4} (1+q^j)\right.\nonumber\\
&\hphantom{XXXXXXXXXXXX}\left. + q^{j(3j-1)/2 - m(21m+17)/2-2} (1+q^j)\right\}
\nonumber\\
&=
\sum_{j=1}^\infty 
\sum_{\substack{-j \le 3m \le j-1}}
\sgn(m) (-1)^{m+j+1} q^{ j(3j-1)/2 - m(21m+11)/2-1 }(1+q^{j})\nonumber\\
& \qquad +
\sum_{j=2}^\infty 
\sum_{\substack{-j-1 \le 3m \le j-2}}
\sgn(m) (-1)^{m+j+1} q^{ j(3j-1)/2 - m(21m+17)/2-2}(1+q^{j}),
\nonumber
\end{align}
which is \eqn{F2id}.
This completes the proof of Theorem \refthm{mock7}.

\section{Zagier's Mock Theta Function Identities and Related Results}
\label{sec:zagier}

In this section we write our double-series identities for the two fifth
order functions $\chi_0(q)$ and $\chi_1(q)$ and all three seventh
order functions $\mathcal{F}_j(q)$ ($j=0,1,2$) using Dirichlet characters.
This leads naturally to relations between the coefficients of these series 
as in Theorems \refthm{chirels} and \refthm{mock7rels}.

As mentioned before Andrews \cite{An86} obtained indefinite theta series  
identities
for all of Ramanujan's fifth order functions except $\chi_0(q)$
and $\chi_1(q)$. Using Andrews's results Zwegers \cite{Zw-thesis}
showed how to complete
all of Andrews's fifth order functions to weak harmonic Maass forms.
As noted by Watson \cite[pp.277-279]{Wa36b}, 
Ramanujan gave identities for $\chi_0(q)$
and $\chi_1(q)$ in terms of the other fifth order functions. Zagier
suggested  that indefinite theta function identities for $\chi_0(q)$ and $\chi_1(q)$
could be obtained from Ramanujan's results and Zwegers transformation
formulas, although he gave no details. 
We state Zagier's results in a modified form in the following
\begin{thm}
\label{thm:zagchithm}
\beqs             
(q)_\infty (2 - \chi_0(q))
= 
\sum_{\substack{ 5\abs{b} < \abs{a} \\ a+b\equiv 2\pmod{4}
\\ a\equiv 2 \pmod{5}}}
(-1)^a \sgn(a) \leg{-3}{a^2-b^2} 
  q^{\frac{1}{120}(a^2 - 5b^2) - \frac{1}{30}}
\eeqs              
and
\beqs                 
(q)_\infty \chi_1(q)
= 
\sum_{\substack{ 5\abs{b} < \abs{a} \\ a+b\equiv 2\pmod{4}
\\ a\equiv 4 \pmod{5}}}
(-1)^a \sgn(a) \leg{-3}{a^2-b^2} 
  q^{\frac{1}{120}(a^2 - 5b^2) - \frac{19}{30}}
\eeqs               
\end{thm}
\begin{rem}
Here $\leg{-3}{\cdot}$ is the Kronecker symbol, and is a Dirichlet character mod $3$.
\end{rem}

Our Theorem \refthm{chi01} seems to differ from Zagier's Theorem.
In contrast to Zagier's theorem which involves a character mod $3$ our version
involves the Dirichlet  character mod $60$:  
$$
\chi_{60}(m) =
\begin{cases}
1 & \mbox{if $m\equiv 1$, $11$, $19$, $29 \pmod{60}$}\\
i & \mbox{if $m\equiv 7$, $13$, $17$, $23 \pmod{60}$}\\
-1 & \mbox{if $m\equiv 31$, $41$, $49$, $59 \pmod{60}$}\\
-i & \mbox{if $m\equiv 37$, $43$, $47$, $53 \pmod{60}$}
\end{cases}
$$

\begin{thm}
\label{thm:chi60}
\beq
\label{eq:chi0altid} 
(q)_\infty (2 - \chi_0(q))
= 
\sum_{\substack{ 3\abs{b} < 5 \abs{a} \\ a\equiv 1\pmod{6}
\\ b\equiv 1,11 \pmod{30}}}
\sgn(b) \leg{12}{a} \chi_{60}(b)
  q^{\frac{1}{120}(5a^2 - b^2) - \frac{1}{30}}
\eeq               
and
\beq                    
\label{eq:chi1altid}
(q)_\infty \chi_1(q)  
= i\,
\sum_{\substack{ 3\abs{b} < 5 \abs{a} \\ a\equiv b\equiv 1\pmod{6}
\\ b\equiv \pm 2 \pmod{5}}}
\sgn(b) \leg{12}{a} \chi_{60}(b)
  q^{\frac{1}{120}(5a^2 - b^2) - \frac{19}{30}}
\eeq                  
\end{thm}

We find analogous identities for the seventh order functions.
Also Andrews \cite{An86} obtained indefinite theta series  identities
for these functions. Hickerson \cite[Theorem 2.0,p.666]{Hi88b} found
nice versions of Andrews identities, which he used to prove
his seventh order analogues of Ramanujan's mock theta conjectures \cite{An-Ga89}
for the fifth order functions. Our identities differ from Andrews's and Hickerson's
and appear to be new.
\begin{thm}
\label{thm:mock7alt}
\beq
\label{eq:F0altid}
(q)_\infty \mathcal{F}_0(q) =
\sum_{\substack{ 3\abs{b} < 7 \abs{a} \\ a\equiv 1\pmod{6}
\\ b\equiv 1,13 \pmod{42}}}
\sgn(b) \leg{12}{a} \leg{12}{b}\leg{b}{7}
  q^{\frac{1}{168}(7a^2 - b^2) - \frac{1}{28}},
\eeq
\beq
\label{eq:F1altid}
(q)_\infty \mathcal{F}_1(q) =
-\sum_{\substack{ 3\abs{b} < 7 \abs{a} \\ a\equiv 1\pmod{6}
\\ b\equiv 5,19 \pmod{42}}}
\sgn(b) \leg{12}{a} \leg{12}{b}\leg{b}{7}
  q^{\frac{1}{168}(7a^2 - b^2) + \frac{3}{28}},
\eeq
and
\beq
\label{eq:F2altid}
(q)_\infty \mathcal{F}_2(q) =
-\sum_{\substack{ 3\abs{b} < 7 \abs{a} \\ a\equiv 1\pmod{6}
\\ b\equiv 11,17 \pmod{42}}}
\sgn(b) \leg{12}{a} \leg{12}{b}\leg{b}{7}
  q^{\frac{1}{168}(7a^2 - b^2) - \frac{9}{28}}.
\eeq
\end{thm}

We sketch the proof of \eqn{chi1altid}. Firstly we observe that
\begin{align*}
\tfrac{1}{2}j(3j\pm1) - \tfrac{1}{2}m(15m+7)-1 
&= \tfrac{1}{120}\Parans{5(6j  \pm1)^2 - (30m+7)^2} - \tfrac{19}{30},\\
\tfrac{1}{2}j(3j\pm1) - \tfrac{1}{2}m(15m+13)-2 
&= \tfrac{1}{120}\Parans{5(6j  \pm1)^2 - (30m+13)^2}- \tfrac{19}{30}.
\end{align*}
In the summations in equation \eqn{chi1altid}, we let $a=6(\pm j)+1$,
and $b=30m+r$, where $j\ge1$, $m\in\mathbb{Z}$, and $r=7$, $13$. We have 
\begin{align*}
\leg{12}{a} &= \leg{12}{6(\pm j)+1} = (-1)^j,\\
i\chi_{60}(b) &= i\chi_{60}(30m+r)=(-1)^{m+1},\quad\mbox{and}\\
\sgn(b) &=\sgn(30m+r)=\sgn(m).
\end{align*}
Next we consider the inequalities for the variables in the summations.

\subsubsection*{Case 1} $m\ge0$ and $r=7$. Then we see that
$$
3\abs{b} < 5\abs{a} \Leftrightarrow 3m < j + \Parans{\frac{\pm5-21}{30}}
\Leftrightarrow 3m \le j-1.
$$
\subsubsection*{Case 2} $m<0$ and $r=7$. Then we see that
$$
3\abs{b} < 5\abs{a} \Leftrightarrow -j < 3m + \Parans{\frac{\pm5+21}{30}}
\Leftrightarrow -j \le 3m.
$$
\subsubsection*{Case 3} $m\ge0$ and $r=13$. Then we see that
$$
3\abs{b} < 5\abs{a} \Leftrightarrow 3m < j + \Parans{\frac{\pm5-39}{30}}
\Leftrightarrow 3m \le j-2.
$$
\subsubsection*{Case 4} $m<0$ and $r=13$. Then we see that
$$
3\abs{b} < 5\abs{a} \Leftrightarrow -j + \Parans{\frac{-39\pm5}{30}}< 3m
\Leftrightarrow -j-1 \le 3m.
$$
It follows that
\begin{align*}       
&\sum_{j=1}^\infty 
\sum_{\substack{-j \le 3m \le j-1}}
\sgn(m) (-1)^{m+j+1} q^{ j(3j-1)/2 - m(15m+7)/2 -1}(1+q^{j})\\
&\qquad
=i\,
\sum_{\substack{ 3\abs{b} < 5 \abs{a} \\ a\equiv1\pmod{6}
\\ b\equiv 7 \pmod{30}}}
\sgn(b) \leg{12}{a} \chi_{60}(b)
  q^{\frac{1}{120}(5a^2 - b^2) - \frac{19}{30}},
\end{align*}
and
\begin{align*}
&\sum_{j=1}^\infty 
\sum_{\substack{-j-1 \le 3m \le j-2}}
\sgn(m) (-1)^{m+j+1} q^{ j(3j-1)/2 - m(15m+13)/2 -2}(1+q^{j})\\
&\qquad
=i\,
\sum_{\substack{ 3\abs{b} < 5 \abs{a} \\ a\equiv1 \pmod{6}
\\ b\equiv 13 \pmod{30}}}
\sgn(b) \leg{12}{a} \chi_{60}(b)
  q^{\frac{1}{120}(5a^2 - b^2) - \frac{19}{30}}.
\end{align*}
Therefore we see that equation \eqn{chi01b} implies \eqn{chi1altid}.                    
The proof of the remaining parts of Theorems \refthm{chi60} and \refthm{mock7alt}
are analogous.

Theorems \refthm{chi60} and \refthm{mock7alt} imply simple relations between
the coefficients. We define the coefficients $C_0(n)$ and $C_1(n)$ by
\begin{align*}
\sum_{n=0}^\infty C_0(n) q^n &= (q)_\infty \,(2 - \chi_0(q)),\\
\sum_{n=0}^\infty C_1(n) q^n &= (q)_\infty \,\chi_1(q),   
\end{align*}
define
\beq
\varepsilon_p =
\begin{cases} 
-1 &\mbox{if $p\equiv 3 \pmod{10}$},\\
1 &\mbox{if $p\equiv 7 \pmod{10}$},
\end{cases}
\eeq
and for an integer $n$ and a prime $p$, define $\nu_p(n)$ to be the
exact power of $p$ dividing $n$.
\begin{thm}
\label{thm:chirels}
If $p>5$ is any prime congruent to $3$ or $7$ mod $10$, then
\begin{align}
C_0(n) &= 0  \qquad \mbox{if $\nu_p(30n+1)=1$},
\label{eq:chirel1}\\
C_0(p^2 n + \tfrac{1}{30}(19p^2-1)) &= - \varepsilon_p\, C_1(n)
\qquad \mbox{for $n \ge 0$}, 
\label{eq:chirel2}\\
C_1(n) &= 0  \qquad \mbox{if $\nu_p(30n+19)=1$},
\label{eq:chirel3}\\
C_1(p^2 n + \tfrac{1}{30}(p^2-19)) &=  \varepsilon_p\, C_0(n)
\qquad \mbox{for $n \ge 0$}.
\label{eq:chirel4}
\end{align}
\end{thm}
\begin{proof}
Suppose $p>5$ is any prime congruent to  $3$ or $7$ mod $10$. Then 
$5$ is a quadratic nonresidue mod $p$. 
Therefore $5a^2-b^2\equiv0\pmod{p}$ implies that $a\equiv b\equiv0\pmod{p}$
and \eqn{chirel1} clearly follows from \eqn{chi0altid}.
Similarly \eqn{chirel3} follows from \eqn{chi1altid}.

We suppose $a\equiv 1\pmod{6}$, $b\equiv1$, $11\pmod{30}$, $3\abs{b}< 5\abs{a}$,
and $a\equiv b\equiv0\pmod{p}$. Letting $a=pa'$, $b=pb'$ we have the following
table
\begin{center}
\begin{tabular}{c c c}
$p\pmod{30}$ & $a'\pmod{6}$ & $b'\pmod{30}$\\
$7$          & $1$          & $13,-7$ \\
$13$          & $1$          & $7,-13$ \\
$17$          & $-1$          & $-7,13$ \\
$23$          & $-1$          & $-13,7$ 
\end{tabular}
\end{center}
By considering the table and noting that the
summation term 
$$
\sgn(b) \leg{12}{a} \chi_{60}(b)
  q^{\frac{1}{120}(5a^2 - b^2) - \frac{1}{30}}
$$
is invariant under both $a\mapsto -a$ and $b\mapsto -b$ we see that
\begin{align*}
&\sum_{n=0}^\infty 
C_0(p^2 n + \tfrac{1}{30}(19p^2-1))\,
q^{p^2 n + \tfrac{1}{30}(19p^2-1)}\\
&\qquad=  
\sum_{\substack{ 3\abs{b'} < 5 \abs{a'} \\ a'\equiv 1\pmod{6}
\\ b'\equiv 7,13 \pmod{30}}}
\sgn(pb') \leg{12}{pa'} \chi_{60}(pb')
  q^{\frac{1}{120}(p^2(5(a')^2 - (b')^2) - \frac{1}{30}}
\\
&
\qquad=
\leg{12}{p} \chi_{60}(p) \sum_{\substack{ 3\abs{b} < 5 \abs{a} \\ a\equiv 1\pmod{6}
\\ b\equiv 7,13 \pmod{30}}}
\sgn(b) \leg{12}{a} \chi_{60}(b)
  q^{\frac{1}{120}(p^2(5a^2 - b^2)) - \frac{1}{30}}
\end{align*}

and

\begin{align*}
&\sum_{n=0}^\infty 
C_0(p^2 n + \tfrac{1}{30}(19p^2-1))\,
q^{n}\\
&\qquad=  
-i\,\varepsilon_p\,\sum_{\substack{ 3\abs{b} < 5 \abs{a} \\ a\equiv 1\pmod{6}
\\ b\equiv 7,13 \pmod{30}}}
\sgn(b) \leg{12}{a} \chi_{60}(b)
  q^{\frac{1}{120}((5a^2 - b^2) - \frac{19}{30}}\\
&\qquad = -\varepsilon_p  (q)_\infty \,\chi_1(q)   
 = -\varepsilon_p \sum_{n=0}^\infty C_1(n) q^n,   
\end{align*}
and \eqn{chirel2} follows.
The proof of \eqn{chirel3}--\eqn{chirel4} is analogous.
\end{proof}

In a similar fashion, Theorem \refthm{mock7alt} implies relations between the
coefficients of the seventh order mock theta functions. For $j=0$, $1$, $2$
we define $f_j(n)$ by
$$                    
\sum_{n=0}^\infty f_j(n) q^n = (q)_\infty \,\mathcal{F}_j(q).
$$
\begin{thm}
\label{thm:mock7rels}
Let $p$ be any odd for which $7$ is 
a quadratic nonresidue mod $p$;i.e.\ $p\equiv\pm5, \pm 11$ or $\pm 13\pmod{28}$.
\begin{enumerate}
\item
Then
\begin{align*}
f_0(n) &= 0  \qquad \mbox{if $\nu_p(28n+1)=1$},\\
f_1(n) &= 0  \qquad \mbox{if $\nu_p(28n-3)=1$},\\
f_2(n) &= 0  \qquad \mbox{if $\nu_p(28n+9)=1$}.
\end{align*}
\item If $p\equiv \pm 5 \pmod{28}$ then
\begin{align*}
f_0(p^2n + \tfrac{1}{28}(9p^2-1)) &= \pm f_2(n),\\
f_1(p^2n + \tfrac{1}{28}(p^2+3)) &= \pm f_0(n),\\
f_2(p^2n + \tfrac{1}{28}(25p^2-9)) &= \mp f_1(n+1).
\end{align*}
\item If $p\equiv \pm 11 \pmod{28}$ then
\begin{align*}
f_0(p^2n + \tfrac{1}{28}(25p^2-1)) &= \mp f_1(n+1),\\
f_1(p^2n + \tfrac{1}{28}(9p^2+3)) &= \pm f_2(n),\\
f_2(p^2n + \tfrac{1}{28}(p^2-9)) &= \mp f_0(n).
\end{align*}
\item If $p\equiv \pm 13 \pmod{28}$ then
\begin{align*}
f_0(p^2n + \tfrac{1}{28}(p^2-1)) &= \mp f_0(n),\\
f_1(p^2n + \tfrac{1}{28}(25p^2+3)) &= \mp f_1(n+1),\\
f_2(p^2n + \tfrac{1}{28}(9p^2-9)) &= \mp f_2(n).
\end{align*}
\end{enumerate}
	\end{thm}

We omit the proof of Theorem \refthm{mock7rels}. The proof is analogous to that
of Theorem \refthm{chirels}.



\section{Concluding Remarks}
\label{sec:conclude}

In Theorems \refthm{chi60} and \refthm{mock7alt} we found new identities
for the fifth order mock theta functions $\chi_0(q)$, $\chi_1(q)$
and all three  seventh order mock theta functions 
$\mathcal{F}_0(q)$, $\mathcal{F}_1(q)$, $\mathcal{F}_2(q)$, 
in terms of Hecke-Rogers indefinite binary theta series.               
This suggests the problem of relating these theorems directly to the results
of Zagier (Theorem \refthm{zagchithm}) for the fifth order functions, 
and to the results of Andrews \cite[Theorem 13, pp.132--133]{An86}
and Hickerson \cite[Theorem 2.0,p.666]{Hi88b} for the seventh order
functions.

\subsection*{Acknowledgments}
I would like to thank Chris Jennings-Shaffer and Jeremy Lovejoy
for their  comments and suggestions. Also I would like to thank the
referee for corrections and suggestions.

\end{document}